\documentclass[12pt]{amsart}
\usepackage{amssymb,bbold}
\usepackage[all]{xy}

\theoremstyle{plain}
\newtheorem{theorem}{Theorem}

\newtheorem{proposition}[theorem]{Proposition}

\theoremstyle{definition}
\newtheorem{definition}[theorem]{Definition}
\newtheorem{remark}[theorem]{Remark}
\newtheorem{example}[theorem]{Example}

\begin{document}
\title[Bounded bigroup homomorphisms]
{Topological modulus of bounded bigroup homomorphisms on a topological module}
\author{Omid Zabeti}
\address{O. Zabeti, Department of Mathematics, Faculty of Mathematics, University of Sistan and Baluchestan,  P.O. Box 98135-674, Zahedan (Iran)}
\email{o.zabeti@gmail.com}
\maketitle
\begin{abstract}
Let X, Y, and Z be  topological modules over a topological ring $R$. In this paper, we introduce three different classes of bounded bigroup homomorphisms from $X\times Y$ into $Z$ with respect to the three different uniform convergence topologies. We show that the operations of addition and module multiplication are continuous for each class of bounded bigroup homomorphisms. Also, we investigate whether each class of bounded bigroup homomorphisms is uniformly complete.
\end{abstract}
\section{introduction and preliminaries}
In \cite{Mi}, it has been introduced some notions for bounded group homomorphisms on a topological ring. Also, it has been proved that each class of bounded group homomorphisms on a topological ring, with respect to an appropriate topology, forms a topological ring. Since every topological ring can be viewed as a topological module over itself, it is routine to see that we can consider the concepts of topological modules of bounded group homomorphisms on a topological module. In fact, the results in \cite{Mi}, can be generalized to topological modules in a natural way. Recall that a \textit{topological module} $X$ is a module over a topological ring $R$ such that the addition ( as a map from $X\times X$ into $X$), and the multiplication ( as a map from $R\times X$ into $X$) are continuous. There are many examples of topological modules, for instance, every topological vector space is a topological module over a topological field, every abelian topological group is a topological module over $\Bbb Z$, where $\Bbb Z$ denotes the ring of integers with discrete topology, and also every topological ring is a topological module over each of its subrings. So, it is of independent interest if we consider possible relations between algebraic structures of a module and its topological properties.  In the present paper, we are going to consider bounded bigroup homomorphisms between topological modules. We endow each class of bounded bigroup homomorphisms to a uniform convergence topology and we show that under the assumed topology, each class of them, forms a topological module. At the end, we see that if each class of bounded bigroup homomorphisms is uniformly complete. In the following, by a bigroup homomorphism on Cartesian product $X\times Y$, we mean a map which is group homomorphism on $X$ and $Y$, respectively. Also, note that if $X$ is a topological module over topological ring $R$, then, $B\subseteq X$ is said to be bounded if for each zero neighborhood $W\subseteq X$, there exists zero neighborhood $V\subseteq R$ such that $RB\subseteq W$. For more information about topological modules, topological rings and the related notions, see \cite{Arnautov, Shell, Tr, Omid}.
\section{bounded bigroup homomorphisms}
\begin{definition}
Let $X$, $Y$, and $Z$ be topological modulus over a topological ring $R$. A bigroup homomorphism $\sigma:X \times Y \to Z$ is said to be:
\begin{itemize}
\item[\em i.]{$n$-bounded if there exist some zero neighborhoods $U\subseteq X$ and $V\subseteq Y$ such that  $\sigma(U,V)$ is bounded in $Z$}.
\item[\em ii.]{${\frac{n}{2}}$-bounded if there exists a zero neighborhood $U\subseteq X$ such that for each bounded set $B\subseteq Y$ $\sigma(U,B)$ is bounded in $Z$}.
\item[\em iii.]{$b$-bounded if for every bounded subsets $ B_1\subseteq X$ and $B_2\subseteq Y$, $\sigma(B_1,B_2)$ is bounded in $Z$}.
\end{itemize}
\end{definition}
The first point is that these concepts of bounded bigroup homomorphisms are far from being equivalent. In prior to anything, we show this.
\begin{example}\label{1}
Let $X={\Bbb R}^{\Bbb N}$, the space of all real sequences, with
the coordinate-wise topology and the pointwise product. It is easy to see that $X$ is a topological module over itself. Consider
the bigroup homomorphism $\sigma:X\times X\to X$ defined by $\sigma(x,y)=xy$, in which the product is given by pointwise. It is not difficult to see that $\sigma$ is $b$-bounded but since $X$ is not locally bounded, it can not be $n$-bounded.
\end{example}
Also, the above example may apply to determine a $b$-bounded bigroup homomorphism which is not $\frac{n}{2}$-bounded.
\begin{example}
Let $X$ be $l^{\infty}$, the space of all bounded real sequences, with the topology induced by uniform norm and pointwise product. Suppose $Y$ is  $l^{\infty}$, with the coordinate-wise topology and pointwise product. Consider the bigroup homomorphism $\sigma$ from $X \times Y$ to $Y$ as in Example \ref{1}. It is easy to see that $\sigma$ is $\frac{n}{2}$-bounded but it is not $n$-bounded. For, suppose $\varepsilon>0$ is arbitrary. Assume that $N_{\varepsilon}^{(0)}$ is the ball with centre zero and radius $\varepsilon$ in $X$. If $U$ is an arbitrary zero neighborhood in $Y$, with out loss of generality, we may assume that $U$ is of the form
\[(-\varepsilon_1,\varepsilon_1)\times\ldots\times(-\varepsilon_r,\varepsilon_r)\times \Bbb R\times \Bbb R\times\ldots,\]
in which, $\varepsilon_i>0$. Fix $0<\delta<\min\{\varepsilon_i\}$. Consider the sequence $(a_n)\subseteq U$ defined by $a_n=(\delta,\ldots,\delta,r+1,\ldots,r+n,o,\ldots)$ for $n>r$ and zero for $n\leq r$. Now, it is not difficult to see that $\sigma(N_{\varepsilon}^{(0)},(a_n))$ can not be a bounded subset of $Y$.
\end{example}
\begin{example}\label{2}
Let $X$ be $\l^{\infty}$, with pointwise product and the uniform norm topology, and $Y$ be $\l^{\infty}$, with the zero multiplication and the topology induced by norm. Consider $\sigma$ from $X\times Y$ to $X$ as in Example \ref{1}. Then, $\sigma$ is $n$-bounded but it is not $\frac{n}{2}$-bounded. For, suppose $\varepsilon>0$ is arbitrary. Consider the sequence $(a_n)$ in $Y$ defined by $a_n=(\frac{1}{\varepsilon},\ldots,\frac{n}{\varepsilon},0,\ldots)$. $(a_n)$ is bounded in $Y$ but $\sigma(N_{\varepsilon}^{(0)},(a_n))$ contains the sequence $(1,\ldots,n,0,\ldots)$ which is not bounded in $X$.
\end{example}
Since topological modules are topological spaces, we can consider the concept of jointly continuity for a bigroup homomorphism between
topological modulus. The interesting result in this case is that there is no relation between jointly continuous bigroup homomorphisms and bounded ones.
To see this, consider the following examples.
\begin{example}
Let $X$ be $\l^{\infty}$, with the pointwise product and coordinate-wise topology, and $Y$ be $\l^{\infty}$, with the zero multiplication and uniform norm topology. Consider the bigroup homomorphism $\sigma$ from $X\times X$ into $Y$ as in Example \ref{1}. Indeed, $\sigma$ is $b$-bounded and $n$-bounded but it is easy to see that $\sigma$ can not be jointly continuous.
\end{example}
The class of all $n$-bounded bigroup homomorphisms on a topological module $X$ is denoted by $B_{n}(X\times X)$ and is equipped with the topology of uniform
convergence on some zero neighborhoods, namely, a net $(\sigma_{\alpha})$ of $n$-bounded bigroup homomorphisms converges uniformly to zero on some zero neighborhoods
$U, V \subseteq X$ if for each zero neighborhood $W\subseteq X$ there is an $\alpha_0$
with $\sigma_{\alpha}(U,V)\subseteq W$ for each $\alpha\geq\alpha_0$.
The set of all $\frac{n}{2}$-bounded bigroup homomorphisms on a topological module $X$ is denoted by $B_{\frac{n}{2}}(X\times X)$ and it is assigned
with the topology of $\sigma$-uniformly convergence on some zero neighborhood. We say that a net $(\sigma_{\alpha})$ of $\frac{n}{2}$-bounded bigroup homomorphisms converges $\sigma$-uniformly to zero on some zero neighborhood if there exists a zero neighborhood $U\subseteq X$ such that for each zero neighborhood $W\subseteq X$ and each bounded set $B\subseteq X$ there is an $\alpha_0$ with  $\sigma_{\alpha}(U,B)\subseteq W$ for each $\alpha\geq\alpha_0$. Finally, the class of all $b$-bounded bigroup homomorphisms on a topological module $X$ is denoted by $B_{b}(X\times X)$ and is endowed with the topology of uniform convergence on bounded sets which means a net $(\sigma_{\alpha})$ of $b$-bounded bigroup homomorphisms converges uniformly to zero on bounded sets $B_1,B_2\subseteq X$ if for each zero neighborhood $W\subseteq X$ there is an $\alpha_0$
with $\sigma_{\alpha}(B_1,B_2)\subseteq W$ for each $\alpha\geq\alpha_0$.
In this part of the paper, we show that the operations of addition and module multiplication are continuous in each of the topological modules $B_n(X\times X), B_{\frac{n}{2}}(X\times X)$, and $B_b(X\times X)$ with respect to the assumed topology, respectively. So, each of them forms a topological $R$-module.
\begin{theorem}
The operations of addition and module multiplication in $B_{n}(X\times X)$ are continuous with respect to the topology of uniform convergence on some zero neighborhoods.
\end{theorem}
\begin{proof}
Suppose two nets $(\sigma_{\alpha})$ and $(\gamma_{\alpha})$  of $n$-bounded bigroup homomorphisms converge to zero uniformly on some zero neighborhoods $(U,V)\subseteq X$. Let $W$ be an arbitrary
zero neighborhood in $X$. So, there is a zero neighborhood
$W_1$ with $ W_1+W_1 \subseteq W$. There are some $\alpha_0$ and $\alpha_1$ such that $\sigma_{\alpha}(U,V) \subseteq W_1$ for each $\alpha\geq\alpha_0$ and $\gamma_{\alpha}(U,V) \subseteq W_1$ for each $\alpha\geq\alpha_1$. Choose an $\alpha_2$ with $\alpha_2\geq \alpha_0$ and $\alpha_2\geq\alpha_1$. If $\alpha\geq\alpha_2$ then $(\sigma_{\alpha}+\gamma_{\alpha})(U,V) \subseteq \sigma_{\alpha}(U,V)+\gamma_{\alpha}(U,V) \subseteq W_1+W_1 \subseteq W$. Thus, the addition is continuous.
Now, we show the continuity of the module multiplication. Suppose $(r_{\alpha})$ is a net in $R$ which is convergent to zero. There are some neighborhoods $V_1\subseteq R$ and $W_2\subseteq X$ such that $V_1W_2\subseteq W$. Find an $\alpha_3$ with $\gamma_{\alpha}(U,V) \subseteq W_2$ for each $\alpha\geq \alpha_3$. Take an $\alpha_4$ such that $(r_{\alpha})\subseteq V_1$ for each $\alpha\geq \alpha_4$. Choose an $\alpha_5$ with $\alpha_5\geq\alpha_3$ and $\alpha_5\geq \alpha_4$. If $\alpha \geq \alpha_5$ then $r_{\alpha}\sigma_{\alpha}(U,V)\subseteq V_1W_2\subseteq W$, as asserted.
\end{proof}
\begin{theorem}
The operations of addition and module multiplication in $B_{\frac{n}{2}}(X\times X)$ are continuous with respect to the topology of $\sigma$-uniform convergence on some zero neighborhood.
\end{theorem}
\begin{proof}
Suppose two nets $(\sigma_{\alpha})$ and $(\gamma_{\alpha})$  of ${\frac{n}{2}}$-bounded bigroup homomorphisms converge to zero $\sigma$-uniformly on some zero neighborhood $U\subseteq X$. Fix a bounded set $B\subseteq X$. Let $W$ be an arbitrary
zero neighborhood in $X$. So, there is a zero neighborhood
$W_1$ with $ W_1+W_1 \subseteq W$. There are some $\alpha_0$ and $\alpha_1$ such that $\sigma_{\alpha}(U,B) \subseteq W_1$ for each $\alpha\geq\alpha_0$ and $\gamma_{\alpha}(U,B) \subseteq W_1$ for each $\alpha\geq\alpha_1$. Choose an $\alpha_2$ with $\alpha_2\geq \alpha_0$ and $\alpha_2\geq\alpha_1$. If $\alpha\geq\alpha_2$ then $(\sigma_{\alpha}+\gamma_{\alpha})(U,B) \subseteq \sigma_{\alpha}(U,B)+\gamma_{\alpha}(U,B) \subseteq W_1+W_1 \subseteq W$. Thus, the addition is continuous.
Now, we show the continuity of the module multiplication. Suppose $(r_{\alpha})$ is a net in $R$ which is convergent to zero. There are some neighborhoods $V_1\subseteq R$ and $W_2\subseteq X$ such that $V_1W_2\subseteq W$. Find an $\alpha_3$ with $\gamma_{\alpha}(U,B) \subseteq W_2$ for each $\alpha\geq \alpha_3$. Take an $\alpha_4$ such that $(r_{\alpha})\subseteq V_1$ for each $\alpha\geq \alpha_4$. Choose an $\alpha_5$ with $\alpha_5\geq\alpha_3$ and $\alpha_5\geq \alpha_4$. If $\alpha \geq \alpha_5$ then $r_{\alpha}\sigma_{\alpha}(U,B)\subseteq V_1W_2\subseteq W$, as we wanted.
\end{proof}
\begin{theorem}
The operations of addition and module multiplication in $B_{b}(X\times X)$ are continuous with respect to the topology of uniform convergence on bounded sets.
\end{theorem}
\begin{proof}
Suppose two nets $(\sigma_{\alpha})$ and $(\gamma_{\alpha})$  of $b$-bounded bigroup homomorphisms converge to zero uniformly on bounded sets. Fix two bounded sets $B_1,B_2\subseteq X$. Let $W$ be an arbitrary
zero neighborhood in $X$. So, there is a zero neighborhood
$W_1$ with $ W_1+W_1 \subseteq W$. There are some $\alpha_0$ and $\alpha_1$ such that $\sigma_{\alpha}(B_1,B_2) \subseteq W_1$ for each $\alpha\geq\alpha_0$ and $\gamma_{\alpha}(B_1,B_2) \subseteq W_1$ for each $\alpha\geq\alpha_1$. Choose an $\alpha_2$ with $\alpha_2\geq \alpha_0$ and $\alpha_2\geq\alpha_1$. If $\alpha\geq\alpha_2$ then $(\sigma_{\alpha}+\gamma_{\alpha})(B_1,B_2) \subseteq \sigma_{\alpha}(B_1,B_2)+\gamma_{\alpha}(B_1,B_2) \subseteq W_1+W_1 \subseteq W$. Thus, the addition is continuous.
Now, we show the continuity of the module multiplication. Suppose $(r_{\alpha})$ is a net in $R$ which is convergent to zero. There are some neighborhoods $V_1\subseteq R$ and $W_2\subseteq X$ such that $V_1W_2\subseteq W$. Find an $\alpha_3$ with $\gamma_{\alpha}(B_1,B_2) \subseteq W_2$ for each $\alpha\geq \alpha_3$. Take an $\alpha_4$ such that $(r_{\alpha})\subseteq V_1$ for each $\alpha\geq \alpha_4$. Choose an $\alpha_5$ with $\alpha_5\geq\alpha_3$ and $\alpha_5\geq \alpha_4$. If $\alpha \geq \alpha_5$ then $r_{\alpha}\sigma_{\alpha}(B_1,B_2)\subseteq V_1W_2\subseteq W$, as asserted.
\end{proof}
In the final part of paper, we investigate whether each class of bounded bigroup homomorphisms is uniformly complete. The answer for $B_b(X\times X)$ is affirmative but for other cases there exist counterexamples.
\begin{remark}\label{3}
The class $B_{n}(X\times X)$ can contain a Cauchy sequence whose limit is not an $n$-bounded bigroup homomorphism.
Let $X={\Bbb R}^{\Bbb N}$, the space of all real sequences, with the coordinate-wise topology and the pointwise product. Define the bigroup homomorphisms $\sigma_n$ on $X$ as follows:
\[\sigma_n(x,y)=(x_1y_1,\ldots,x_ny_n,0,\ldots),\]
in which $x=(x_i)_{i=1}^{\infty}$ and $y=(y_i)_{i=1}^{\infty}$. Each $\sigma_n$ is $n$-bounded. For, if
\[U_n =\{x \in X, |x_j|<1,j=0,1,\ldots,n\},\]
then, $\sigma_n(U_n,U_n)$ is bounded in $X$. Also, $(\sigma_n)$ is a Cauchy sequence in $B_{n}(X\times X)$. Because if $W$ is an arbitrary zero neighborhood in $X$, without loss of generality, we may assume that it is of the form
\[W=(-\varepsilon_1,\varepsilon_1)\times\ldots\times(-\varepsilon_r,\varepsilon_r)\times \Bbb R\times \Bbb R\times\ldots,\]
in which $\varepsilon_i>0$. So, for $m,n >r$; we have $(\sigma_n-\sigma_m)(X,X)\subseteq W$. Also, $(\sigma_n)$ converges uniformly on $(X,X)$ to the bigroup homomorphism $\sigma$ defined by
\[\sigma(x,y)=(x_1y_1,x_2y_2,\ldots).\]
But we have seen in Example \ref{1} that $\sigma$ is not $n$-bounded.
\end{remark}
\begin{remark}
The class $B_{{\frac{n}{2}}}(X\times X)$ can contain a Cauchy sequence whose limit is not an ${\frac{n}{2}}$-bounded bigroup homomorphism.
Let $X$ be $\l^{\infty}$, with the pointwise product and the uniform norm topology, and $Y$ be $\l^{\infty}$, with the zero multiplication and the topology induced by norm. Consider bigroup homomorphisms $\sigma_n$ from $X\times Y$ to $X$ as in Remark \ref{3}. It is not difficult to see that each $\sigma_n$ is  ${\frac{n}{2}}$-bounded. Also, $(\sigma_n)$ is a Cauchy sequence in $B_{{\frac{n}{2}}}(X\times X)$ which is convergent $\sigma$-uniformly on $X$ to the bigroup homomorphism $\sigma$ described in Example \ref{2}, so that it is not an  ${\frac{n}{2}}$-bounded bigroup homomorphism.
\end{remark}
\begin{proposition}
Suppose  a net $(\sigma_{\alpha})$ of $b$-bounded bigroup homomorphisms converges to a bigroup homomorphism uniformly on bounded sets.
Then $\sigma$ is also $b$-bounded.
\end{proposition}
\begin{proof}
Fix  bounded sets $B_1,B_2\subseteq X$. Let $W$ be an arbitrary zero neighborhood in $X$. There is a zero neighborhood $W_1$ such that $W_1+W_1\subseteq W$. Choose a zero neighborhood $V_1\subseteq R$ and a zero neighborhood $W_2\subseteq X$ with $V_1 W_2 \subseteq W_1$. There is an $\alpha_0$ such that $(\sigma_{\alpha}-\sigma)(B_1,B_2) \subseteq W_2$ for each $\alpha\geq\alpha_0$. Fix an $\alpha\geq\alpha_0$. So, there is a zero neighborhood $V_2\subseteq V_1$ with $V_2\sigma_{\alpha}(B_1,B_2) \subseteq W_2$. Therefore,
\[V_2\sigma(B_1,B_2) \subseteq V_2\sigma_{\alpha}(B_1,B_2) + V_2 W_2 \subseteq W_2+V_1W_2 \subseteq W_1+W_1 \subseteq W.\]
\end{proof}

\end{document}